\documentclass[10pt,twocolumn,twoside]{ilcss}
\usepackage{amsthm}
\templatetype{ilcssworkingpaper} 
\newtheorem{theorem}{Theorem}[section]
\newtheorem{lemma}[theorem]{Claim}

\usepackage{xcolor}

\title{\LARGE Modeling the 15 Tile Puzzle Through the Lens of Group Theory}	

\author[]{Viren Khandal}

\affil[]{University of California, Berkeley}

\begin{abstract}
In this piece, we examine one variant of the infamous 15 Tile Puzzle and develop a mathematical backing behind why it is unsolvable. Using concepts of permutations, bijectivity, and cycle transpositions, we not only prove how to model this puzzle as a \textit{group}, but also how to determine if a certain configuration of the puzzle is reachable.
\end{abstract}
\dates{December 10, 2021}
\begin{document}
\maketitle
\ifthenelse{\boolean{shortarticle}}{\ifthenelse{\boolean{singlecolumn}}{\abscontentformatted}{\abscontent}}{}
\dropcap{T}he 15 tile puzzle has been a prevalent brain teaser in the realm of mathematics for several decades. A new variant, invented in the mid-1870s by Sam Lloyd, was known to \textit{drive the world crazy} as many deemed it impossible to solve. In fact, even though there was a \$1000 prize was rewarded to any solvers, no one claimed the prize, stating that the puzzle was unsolvable.

The below Figure 1 shows an example of the starting position in Lloyd's 15 tile puzzle. The board consists of a 4 by 4 grid, in which each cell but one contains a tile, represented by number ranging from 1 to 15. The goal of the puzzle can be seen as shifting the tiles in an orderly fashion to achieve the goal state, which is displayed in Figure 2. Tiles can be moved vertically and horizontally into the empty space in order to reach the aforementioned goal state.

\begin{figure}[h]
    \centering
    \label{fig:one}
    \includegraphics[width=.25\textwidth]{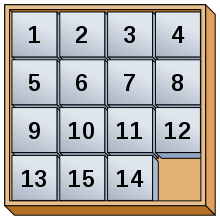}
    \caption{An example of Sam Lloyd's 15 slide puzzle.}
    \centering
    \label{fig:two}
    \includegraphics[width=.25\textwidth]{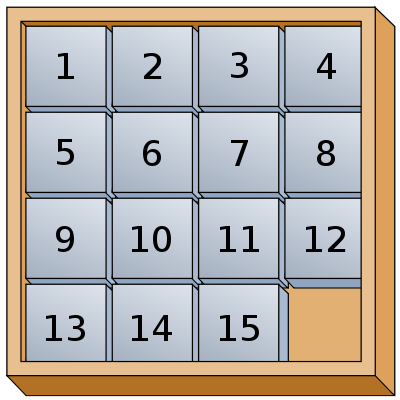}
    \caption{An example of the solved 15 slide puzzle.}
\end{figure}


\section*{Discussion}
In this section, we discuss the connection of the 15 Tile Puzzle to some of the concepts covered in the MATH 113: Introduction to Abstract Algebra course at the University of California, Berkeley.
\\

\noindent
\textit{Note: In this works, the words $mapping$ and $permutation$ are used interchangeably.}
\\\\
We can model the placement of tiles on the 4x4 grid using a mapping C such that:
\begin{equation}
\label{eq:c}
    C: \{1, 2, ..., 16\} \xrightarrow{} \{1, 2, ..., 16\}
\end{equation}

\noindent
Here, C maps each position on the grid labelled 1 through 16 to a tile labelled 1 to 16, where 16 is the empty tile. Using this mapping C, we can rewrite a configuration of the below shown board, $A$, as a permutation in two line and cycle notation:
\begin{figure}[htpb]
    \centering
    \includegraphics[width=.25\textwidth]{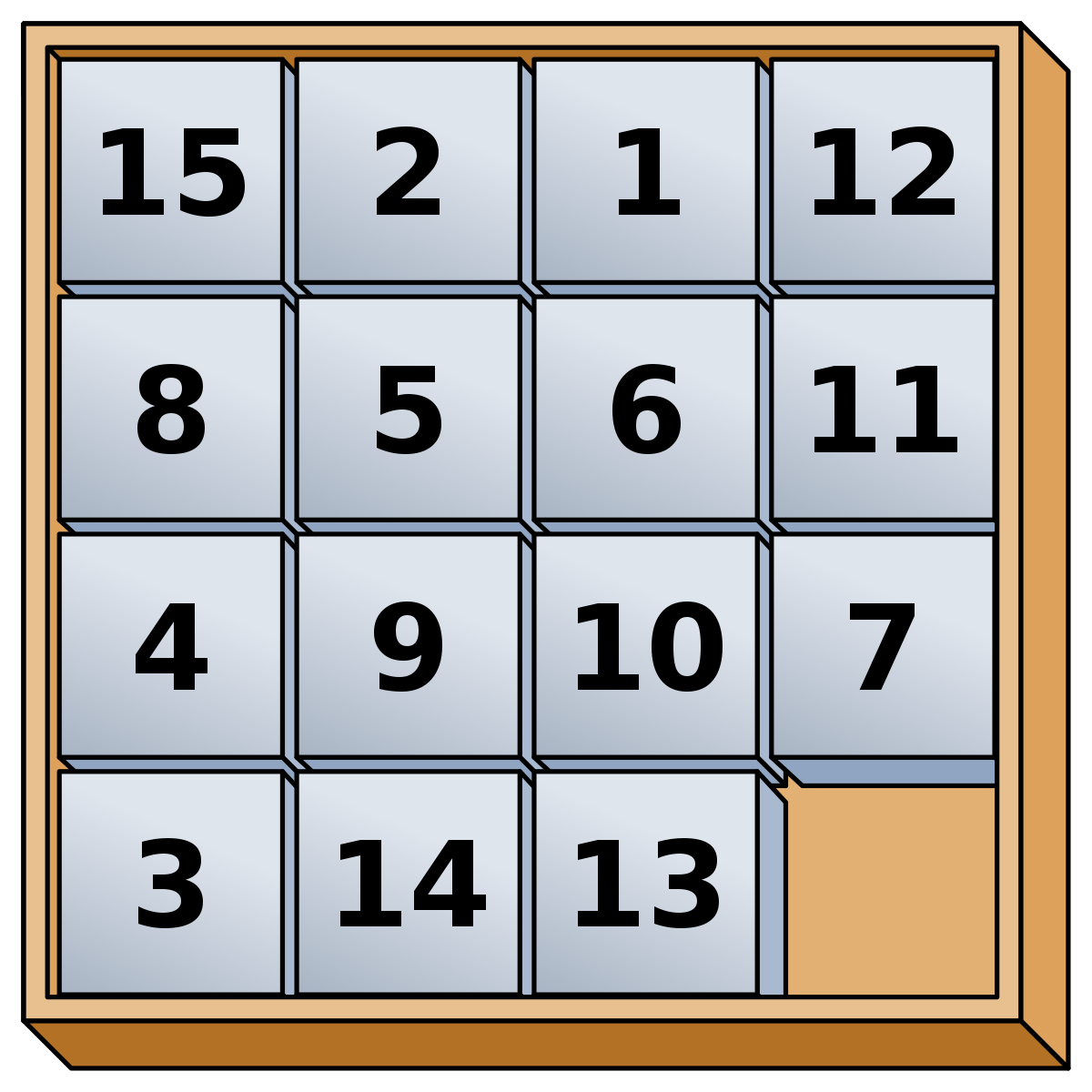}
    \caption{\textbf{A}: An example of a randomly configured board.}
    \label{fig:three}
\end{figure}
\begin{equation}
\begin{aligned}
\label{eq:permutation}
  A =
  \begin{pmatrix}
    {1 \atop 3} \; {2 \atop 2} \; {3 \atop 13} \; {4 \atop 9} \; {5 \atop 6}\; {6 \atop 7}\; {7 \atop 12}\; {8 \atop 5}\; {9 \atop 10}\; {10 \atop 11}\; {11 \atop 8}\; {12 \atop 4}\; {13 \atop 15}\; {14 \atop 14}\; {15 \atop 1}\; {16 \atop 16}
  \end{pmatrix} \\
  A = (1 \; 3 \; 13 \; 15)\;(2)\;(4 \; 9 \; 10 \; 11 \; 8 \; 5 \; 6 \; 7 \; 12)\;(14)\;(16)
\end{aligned}
\end{equation}

\noindent
Further, we can model the movement of tiles as a mapping $M_1$  such that:
\begin{equation}
\label{eq:m}
    M_1: \{1, 2, ..., 16\} \xrightarrow{} \{1, 2, ..., 16\}
\end{equation}

\noindent
Here, $M_1$ maps each tile labelled 1 through 16, where 16 is the empty tile, to a position on the grid labelled 1 through 16. The set of all such mappings, $M$, is a group under composition of moves\footnote{Trappa, Peter (2004) \emph{Permutations and the 15-Puzzle}, University of Utah, Mathematics.}.
\\

\noindent
Using some configuration of the board, $C$, and a permutation of moves, $M$, we can calculate the configuration of the board, $C'$, after $M$ is performed on $C$:

\begin{equation}
\label{eq:update}
  C' = M * C
\end{equation}

\noindent
Starting with some configuration C, we can aggregate a series of moves $M_1, M_2, ..., M_n$ and formulate a new configuration C' using:

\begin{equation}
    \label{eq:moves}
    C' = M_n*...*M_2*M_1*C
\end{equation}

\noindent
One key detail to note is that moving a tile into the position occupied by the empty square has the effect of moving the empty square one spot along to the former location of the tile that was moved.
Thus far, we have formulated both M and C as permutations in $S_{16}$. Phrased in terms of material from MATH 113, the moves performed on a certain configuration is equivalent to defining it as a bijection, which is an element of $S_{16}$, on the configuration.

\begin{lemma}
\label{thm:bij}
Performing a move on a certain configuration of the 15 Tile Puzzle is a bijection on the set of configurations and thus the elements of M are also elements of $S_{16}$.
\end{lemma}

\begin{proof}
To prove a certain operation is bijective, we must first prove it is, both, injective (one-to-one) and surjective (onto). 
\\

\noindent
We begin by proving the injectivity of the operation.
Assume $C_1$ and $C_3$ to be configurations of the puzzle such that $C_1$ $\neq$ $C_3$ and let $\varphi_{M_1}$ be a function that performs move $M_1 \in M$, such that $\varphi_{M_1}(C_1)$ = $\varphi_{M_1}(C_3)$.
In order for $\varphi_{M_1}$ to be injective, we must show that $\varphi_{M_1}(C_1)$ = $\varphi_{M_1}(C_3)$ implies $C_1$ = $C_3$. We know there exists a $M_1^{-1}$ because $M$ is a group, as previously defined.
\begin{equation}
\begin{aligned}
    \varphi_{M_1}(C_1) = M_1 * C_1 = M_1 * C_3 = \varphi_{M_1}(C_3)\\
\label{eq:equal}
    M_1^{-1} * M_1 * C_1 = M_1^{-1} * M_1 * C_3
\end{aligned}
\end{equation}

\noindent
Therefore by \eqref{eq:equal}, it must be that $C_1$ = $C_3$ and so $\varphi_{M_1}$ is injective, or one-to-one.
\\

\noindent
Now that we've shown this mapping is well-defined and injective, we can proceed to proving this operation is surjective, or onto. The proof of surjectivity is trivial because we have already considered $M$ to be a group, meaning that there exists some move $M_1^{-1}$ that can be performed on $\varphi_{M_1}(C_1)$ to reach $C_1$; and so, $\varphi_{M_1}$ is surjective.
\\
\end{proof}

\noindent
We will now transition into explaining why Lloyd's puzzle is truly unsolvable.

\begin{lemma}
Lloyd's 15 Tile Puzzle is unsolvable
\end{lemma}
\begin{proof}
If we simplify the moves needed to reach from our starting configuration (Fig 1) to our goal state (Fig 2) can be simply written as $(14 \; 15)$, since we just need to swap the 14 and 15 tile.
\\\\
\noindent
We can use our implementation of the mapping $M$ \eqref{eq:m} to rewrite $(14 \; 15)$:
\begin{equation}
    \label{eq:permutation}
    (14 \; 15) = M_n * ... * M_2 * M_1
\end{equation}

\noindent
Because the empty tile stays in the same location in the goal state as the ending state, we can claim that tile 16 moves up and down an equal number of times and left and right an equal number of times. This means that the number of transpositions on the right side of \eqref{eq:permutation} is even. However, the left side of our equation has an odd number of transpositions. 
\\\\
\noindent
This means we have a contradiction in that we are equating an even number of transpositions (left side of the equation) to an even number of transpositions (right side of the equation). This contradictory proof shows that it is impossible to reach the goal state (Fig 2) from our starting state (Fig 1).
\\
\end{proof}

\noindent
As such, we can conclude, through rigorous mathematical backing, that Sam Lloyd's variation of the 15 Tile Puzzle was indeed unsolvable. By initially developing a mappings/permutations of configurations and moves on the 4x4 grid to eventually grounding the moves as a series of transpositions, we have developed the mathematical intuition behind Lloyd's Puzzle and why it is unsolvable.
\\

\noindent
Further connections to abstract algebra topics can be found in other variants of the 15 Tile Puzzle, in which we start with a random configuration and we can use similar strategies to not only determine if the puzzle is solvable given that current state, but also develop a unique trajectory or sequence of moves that will allow us to go from our random configuration (i.e. Fig \ref{fig:three}) to the goal state (Fig 2).


\section*{Acknowledgement}
We thank \textbf{Professor Christopher Ryba} and the Mathematics Department of the University of California, Berkeley for providing the adequate mathematical background on these topics to allow for this deeper investigation.


\end{document}